\documentclass[10pt,bezier]{article}
\usepackage{amsmath,amssymb,amsfonts,graphicx,caption,subcaption}
\usepackage[colorlinks,linkcolor=red,citecolor=blue]{hyperref}

\newtheorem{prethm}{{\bf Theorem}}[section]
\newenvironment{thm}{\begin{prethm}{\hspace{-0.5
em}{\bf.}}}{\end{prethm}}
\newtheorem{prepro}{{\bf Theorem}}

\newtheorem{precor}[prethm]{{\bf Corollary}}
\newenvironment{cor}{\begin{precor}{\hspace{-0.5
em}{\bf.}}}{\end{precor}}
\newtheorem{preconj}[prethm]{{\bf Conjecture}}

\newtheorem{preremark}[prethm]{{\bf Remark}}
\newenvironment{remark}{\begin{preremark}\em{\hspace{-0.5
em}{\bf.}}}{\end{preremark}}
\newtheorem{prelem}[prethm]{{\bf Lemma}}
\newenvironment{lem}{\begin{prelem}{\hspace{-0.5
em}{\bf.}}}{\end{prelem}}
\newtheorem{preque}[prethm]{{\bf Question}}

\newtheorem{preobserv}[prethm]{{\bf Observation}}

\newtheorem{preproposition}[prethm]{{\bf Proposition}}

\newtheorem{preproof}{{\bf Proof.}}
\newtheorem{preprooff}{{\bf Proof}}

\newenvironment{proof}[1]{\begin{preproof}{\rm
#1}\hfill{$\Box$}}{\end{preproof}}

\newtheorem{preproofs}{{\bf The second proof of }}

\newtheorem{preprooft}{{\bf The third proof of }}

\newtheorem{preproofF}{{\bf Proof of}}

\title{\bf\Large 
The existence of $m$-tree-connected $(g,f+f'-m)$-factors using $(g,f)$-factors and $m$-tree-connected $(m,f')$-factors
}
\author{{\normalsize{\sc Morteza Hasanvand${}$} }\vspace{3mm}
\\{\footnotesize{${}$\it Department of Mathematical
 Sciences, Sharif
University of Technology, Tehran, Iran}}
{\footnotesize{}}\\{\footnotesize{ $\mathsf{morteza.hasanvand@alum.sharif.edu }$ }}}

\date{}
\begin{document}
\maketitle
\begin{abstract}{
Let $G$ be a graph and let $g$, $f$, and $f'$ be three positive integer-valued functions on $V(G)$ with $g\le f$. Tokuda, Xu, and Wang (2003) showed that if $G$ contains a $(g,f)$-factor and a spanning $f'$-tree, then $G$ also contains a connected $(g,f+f'-1)$-factor. In this note, we develop their result to a tree-connected version by proving that if $G$ contains a $(g,f)$-factor and an $m$-tree-connected $(m,f')$-factor, then $G$ also contains an $m$-tree-connected $(g,f+f'-m)$-factor, provided that $f\ge m$. In addition, we show that $g$ allows to be nonnegative.
} {\small
}
\end{abstract}
%
%
%==============================================================================
%
%
%
%
%
%
%
%
%
%
%
%
%									Introduction
%==============================================================================
\section{Introduction}
%==============================================================================
%									Definitions
In this note, all graphs have no loop, but multiple edges are allowed and a simple graph is a graph without multiple edges.
 Let $G$ be a graph. 
The vertex set, the edge set, and the maximum degree of $G$ are denoted by $V(G)$, $E(G)$, $\Delta(G)$, respectively. 
The degree $d_G(v)$ of a vertex $v$ is the number of edges of $G$ incident to $v$.
The set of edges of $G$ that are incident to $v$ is denoted by $E_G(v)$.
For a set $A$ of integers, an {\it $A$-factor} is a spanning subgraph with vertex degrees in $A$. 
Let $g$ and $f$ be two nonnegative integer-valued functions on $V(G)$.
A {\it $(g,f)$-factor} refers to spanning subgraph $F$ such that for each vertex $v$, $g(v)\le d_F(v)\le f(v)$.
A graph $G$ is called {\it $m$-tree-connected}, if it has $m$ edge-disjoint spanning trees. 
 Throughout this article, all variables $m$ are positive integers.
%
%==============================================================================
%			
%
%
%
%
%==============================================================================
%

Li, Xu, Chen, and Liu~\cite{Li-Xu-Chen-Liu-2005} showed that every graph containing a Hamiltonian path and a $k$-factor must have a connected $\{k,k+1\}$-factor. More generally, they proved the following stronger assertion. Note that by inserting a Hamiltonian path into a $k$-factor, one can only guarantee the existence of a connected $\{k,k+1,k+2\}$-factor.
\begin{thm}{\rm(\cite{Li-Xu-Chen-Liu-2005})}\label{intro:thm:Hamiltonian-path}
{Let $G$ be a graph. If $G$ contains a $(g,f)$-factor and a Hamiltonian path, then it contains a connected $(g,f+1)$-factor,
where $g$ and $f$ are positive integer-valued functions on $V(G)$ with $2\le g\le f$.
}\end{thm}

Later, Tokuda, Xu, and Wang~\cite{Tokuda-Xu-Wang-2003} generalized Theorem~\ref{intro:thm:Hamiltonian-path} to the following version by replacing the assumption of having a spanning $f'$-free.
In Section~\ref{sec:Matching-Tree}, we refine this result by showing that $g$ allows to be a nonnegative integer-valued function.
\begin{thm}{\rm(\cite{Tokuda-Xu-Wang-2003})}\label{intro:thm:f'-tree}
{Let $G$ be a graph. If $G$ contains a $(g,f)$-factor and a spanning $f'$-tree, then $G$ contains a connected $(g,f+f'-1)$-factor, where $g$, $f$, and $f'$ are positive integer-valued functions on $V(G)$ with $g\le f$.
}\end{thm}

Moreover, we generalize their result to the following tree-connected version in Section~\ref{sec:tree-connected}, which can guarantee the existence of an $m$-tree-connected $(g,f+f'-m)$-factor using a $(g,f)$-factor and an $m$-tree-connected $(m,f')$-factor. As a consequence, if a graph $G$ contains a $(g,f)$-factor and two edge-disjoint Hamiltonian paths, then it has a $2$-edge-connected $(g,f+2)$-factor.
\begin{thm}\label{intro:thm:tree-connected}
{Let $G$ be a graph. If $G$ contains a $(g,f)$-factor and an $m$-tree-connected $(m,f')$-factor, then it contains an $m$-tree-connected $(g,f+f'-m)$-factor, where let $g$, $f$, $f'$ are nonnegative integer-valued functions on $V(G)$ with $ g\le f$, $m\le f$, and $m\le f'$.
}\end{thm}
%
%
%==============================================================================
%
%
%
%
%
%
%
%==============================================================================
%
%
\section{Connected factors with bound degrees}
\label{sec:Matching-Tree}
In this section, we devote a stronger version to Theorem~\ref{intro:thm:f'-tree} as the following result
which provides a relationship between connected factors, matchings, and spanning trees. 
This result was implicitly proved in~\cite{Tokuda-Xu-Wang-2003} for simple graphs and factors $F$ with minimum degree at least one. The proof presented here is a revised version of the proof of Theorem 1 in~\cite{Tokuda-Xu-Wang-2003}.
\begin{thm}\label{thm:F:matching}
{Let $G$ be a graph with a factor $F$. 
Let $M$ be a matching of $F$ having exactly one edge of every non-trivial component of $F$ such that the edge is incident to 
a non-cut vertex of $F$. If $T$ is a spanning tree of $G$, then the factor $F\setminus E(M)$ of $G$ can be extended to a connected factor $H$ 
 such that for each vertex~$v$, 
$$d_F(v)\le d_H(v)\le d_T(v)+\max\{0,d_F(v)-1\}.$$
}\end{thm}
\begin{proof}
{We may assume that $G=T\cup F$. 
Set $M=\{x_1y_1,\ldots, x_ty_t\}$, $X=\{x_i:1\le i\le t\}$, and $Y=\{y_i:1\le i\le t\}$.
Assume that each $x_i\in X$ is not a cut vertex in $F$.
For a graph $G'$ and a vertex $v$, we denote by $E_{G'}(v)$ the set of edges of $G'$ that are incident to $v$.
Let $\mathcal{A}$ be the set of all $2$-tuples $(H,T_0)$ such that 
 $H$ is a connected factor of $G$ containing $E(F)\setminus E(M)$ and $T_0$ is a spanning tree of $H$ in which
for each $v\in V(H)$ the following condition holds:
if (i) $d_H(v)=d_T(v)+d_F(v)$ or (ii) $v=x_i$ and $x_iy_i\notin E(H)$, 
then $E_{T_0}(v)\cap E_{F}(v)=\emptyset$.
Note that $(G,T)\in \mathcal{A}$.
Define $\mathcal{A}_0$ to be the set of all $2$-tuples $(H,T_0)\in\mathcal{A} $ with the minimum $|E(H)|$.
 Now, we prove the following claim.
%
%===================================	Claim
\vspace{2mm}\\
{\bf Claim A.} If $(H,T_0)\in \mathcal{A}_0$, then 
for each vertex $v$, $d_H(v)\le d_T(v)+\max\{0,d_F(v)-1\}$.
\vspace{2mm}\\
%===================================
{\bf Proof of Claim A.} 
Suppose, to the contrary, that
 $d_H(u)=d_T(u)+d_F(u)$ and $d_F(u)>0$, for some vertex $u$.
Let $C$ be the non-trivial component of $F$ containing $u$ and let 
$x_iy_i$ be the single edge in $E(C) \cap E(M)$.
For convenience, let us set $x=x_i$ and $y=y_i$.
First, assume that $d_H(v)=d_T(v)+d_F(v)$, for all $v\in V(C)$.
Since $d_H(x)=d_G(x)$, we have $xy\in E(H)$ and 
 also item (i) implies that $E_{T_0}(x)\cap E_F(x)=\emptyset$ and $xy\notin E(T_0)$.
Define $H'=H-xy$ and $T'_0=T_0$. 
Note that $H'$ contains $E(F)\setminus E(M)$, 
$d_{H'}(y)<d_G(y)\le d_T(y)+d_F(y)$, and
$E_{T'_0}(x)\cap E_F(x)=\emptyset$.
It is easy check that $(H',T'_0)\in \mathcal{A}_0$.
Since $|E(H')|=|E(H)|-1$, we arrive at a contradiction.
Now, assume that $d_H(w)< d_T(w)+d_F(w)$, for some vertex $w$ in $V(C)$.
If $xy\in E(H)$, we take $ab$ to be an edge on 
a $wu$-path
 in the connected graph $C$ such that 
$d_H(a)<d_T(a)+d_F(a)$ and $d_H(b)=d_T(b)+d_F(b)$.
If $xy\notin E(H)$, then $d_H(y)<d_G(y)\le d_{T}(y)+d_F(y)$ and also $x\neq u$. 
In this case, we can take $ab$ to be an edge on 
a $yu$-path in the connected graph $C-x$ such that 
$d_H(a)<d_T(a)+d_F(a)$ and $d_H(b)=d_T(b)+d_F(b)$.
In both cases, since $d_H(b)=d_G(b)$, we must have $ab\in E(H)$
and also item (i) implies that $E_{T_0}(b)\cap E_F(b)=\emptyset$ and $ab\notin E(T_0)$. 
Thus there is an edge $bc\in E(T_0)$ such that $T_0-bc +ab$ is connected (we might have $b=c$). 
Define $H'=H-bc$ and $T'_0=T_0-bc+ab$. 
Since $E_{T_0}(b)\cap E_F(b)=\emptyset$, we have $bc\notin E(F)$ and so the graph $H'$ contains $E(F)\setminus E(M)$. 
Moreover, $d_{H'}(a) \le d_H(a)< d_T(a)+d_F(a)$, $d_{H'}(b) < d_G(b)\le d_T(b)+d_F(b)$, 
 and $d_{H'}(c) < d_G(c)\le d_T(c)+d_F(c)$.
According to the construction, if $xy\notin E(H')$ then we must have $xy\notin E(H)$ and $x\notin \{a,b\}$;
moreover, by item (ii), $E_{T_0}(x)\cap E_{F}(x)=\emptyset$ and so $E_{T'_0}(x)\cap E_{F}(x)=\emptyset$.
It is easy to check that $(H',T'_0)\in \mathcal{A}_0$.
Since $|E(H')|= |E(H)|-1$, we again arrive at a contradiction and so the claim holds.
\vspace{2mm}\\
{\bf Claim B.} If $(H,T_0)$ is a $2$-tuple of $\mathcal{A}_0$ with the maximum $|E(H)\cap E(M)|$, then 
for all vertices $v$, $d_H(v)\ge d_F(v)$.
\vspace{2mm}\\
%===================================
{\bf Proof of Claim B.}
Suppose, to the contrary, that
 $d_H(u)<d_F(u)$ for some vertex $u$.
Hence $u$ must be incident to the edges in $E(M)\setminus E(H)$.
We may assume that $u\in \{x_i, y_i\}$ and $x_iy_i\notin E(H)$, and also $d_H(u)= d_F(u)-1$.
If $u=x_i \in X$, then by item (ii), $E_{T_0}(u)\cap E_{F}(u)=\emptyset$ and so
 $d_H(u)\ge d_{F\setminus E(M)}(u)+d_{T_0}(u)\ge d_F(u)$.
Thus $u=y_i\in Y$.
For convenience, let us again set $x=x_i$ and $y=y_i$. 
Since $xy\not\in E(T_0)$, there is an edge $xz\in E(T_0)$ such that $T_0-xz+xy$ is connected. 
Note that we must have $y\neq z$. 
Define $H'=H-xz+xy$ and $T'_0=T_0-xz+xy$.
By item (ii), $E_{T_0}(x)\cap E_F(x)=\emptyset$, and so $xz\notin E(F)$.
This implies that $H'$ contains $E(F)\setminus E(M)$.
Moreover, $d_{H'}(x)<d_G(x)\le d_T(x)+d_F(x)$,
 $d_{H'}(z)<d_G(z)\le d_T(z)+d_F(z)$, and $d_{H'}(y)=d_F(y)<d_T(y)+d_F(y)$, and also
 $xy\in E(H')$.
It is easy check that $(H',T'_0)\in \mathcal{A}_0$. 
Since $|E(H')|= |E(H)|$ and $|E(H')\cap E(M)|=|E(H)\cap E(M)|+1$, we arrive at a contradiction. Hence the claim holds.

Since $\mathcal{A}_0$ is not empty, we can consider $(H,T')$ as a $2$-tuple of $\mathcal{A}_0$ 
with the maximum $|E(H)\cap E(M)|$.
By Claims A and B, the graph $H$ is the desired connected factor we are looking for.
}\end{proof}
The following corollary is a consequence of Lemma 4.1 in~\cite{ClosedWalks} and 
can be proved by Theorem~\ref{thm:F:matching}
with a little extra effort.
\begin{cor}{\rm (\cite{ClosedWalks})}\label{cor:matching:connected-factors}
{If every matching of a graph $G$ can be extended to a spanning $f'$-tree, then
every $(g,f)$-factor can also be extended to a connected $(g,f+f'-1)$-factor, where
$g$ is a nonnegative integer-valued function on $V(G)$, and $f$ and $f'$ are positive integer-valued functions on $V(G)$.
}\end{cor}
\begin{proof}
{Let $F$ be a $(g,f)$-factor of $G$ and consider $M$ as an arbitrary matching of $F$ 
satisfying the assumption of Theorem~\ref{thm:F:matching}.
Note that every non-trivial component of $F$ has a non-cut vertex.
By the assumption, the graph $G$ has a spanning $f'$-tree $T$ containing $M$.
Thus Theorem~\ref{thm:F:matching} implies that $T\cup F$ has a connected factor $H_0$ containing $E(F)\setminus E(M)$ such that for each vertex $v$, $d_{F}(v)\le d_{H_0}(v)\le d_T(v)+\max\{0,d_{F}(v)-1\}$. Define $H=H_0\cup M$. 
According to this definition, $H$ contains $F$ and is connected as well.
 Let $v\in V(H)$. If $v$ is not incident to the edges of $M$, 
then we still have $ d_{H}(v)=d_{H_0}(v) \le d_T(v)+\max\{0,d_{F}(v)-1\}$. Otherwise, since $E(M)\subseteq E(T)\cap E(F)$, we must automatically have $ d_{H}(v) \le d_T(v)+d_F(v)-d_M(v)=d_T(v)+d_F(v)-1$. 
Hence $ H$ is the desired connected $(g,f+f'-1)$-factor we are looking for.
}\end{proof}
 Tokuda, Xu, and Wang~\cite{Tokuda-Xu-Wang-2003} discovered the following result, when $g'$ is a positive function.
\begin{cor}
{Let $G$ be a graph. If $G$ contains a $(g,f)$-factor and a spanning $f'$-tree, then $G$ has a connected $(g,f+f'-1)$-factor, where $g$ is a nonnegative integer-valued function on $V(G)$, and $f$ and $f'$ are positive integer-valued functions on $V(G)$.
}\end{cor}
\begin{proof}
{Let $F$ be a $(g,f)$-factor and let $T$ be a spanning $f'$-tree of $G$.
Consider $M$ as an arbitrary matching of $F$ satisfying the assumption of Theorem~\ref{thm:F:matching}.
Note that every non-trivial component of $F$ has a non-cut vertex. For finding such a vertex, it is enough to select a vertex with degree one from a spanning tree of that component. 
Now, by Theorem~\ref{thm:F:matching}, the graph $G$ has a connected factor $H$ such that for each vertex $v$, $g'(v) \le d_F(v) \le d_H(v)\le d_T(v)+\max\{0,d_F(v)-1\}\le f(v)+\max\{0,f'(v)-1\}=f(v)+f'(v)-1$.
Hence the assertion holds.
}\end{proof}
\section{Highly tree-connected factors with bound degrees}
\label{sec:tree-connected}
In this section, we first devote a similar tree-connected version to Theorem~\ref{thm:F:matching} in order to prove Theorem~\ref{intro:thm:tree-connected}. For this purpose, we need the following lemma which is a useful tool for finding a pair of edges such that replacing them preserves tree-connectivity of a given tree-connected factor (this is a basic tool in matroid theory~\cite{Edmonds-1970}). 
\begin{lem}{\rm (see \cite{ClosedTrails})}\label{lem:xGy-exchange}
{Let $G$ be graph with an $m$-tree-connected factor $H$ and let $xy\in E(G)\setminus E(H)$. If $Q$ is a minimal $m$-tree-connected subgraph of $G$ including $x$ and $y$, then for every $e \in E(Q)$, the graph $H-e+xy$ remains $m$-tree-connected.
}\end{lem}
The following theorem is essential in this section.
\begin{thm}\label{thm:4.13}
{Let $G$ be a graph with a factor $F$ and an $m$-tree-connected factor $T$. If $F\setminus E(T)$ is a bipartite graph with one partite set $A=\{v\in V(G):d_F(v)\le m\}$, 
then $G$ has an $m$-tree-connected factor $H$ containing $F\setminus E(M)$ such that for each vertex $v$, 
$$d_F(v)\le d_H(v)\le d_T(v)+\max\{0,d_F(v)-m\},$$
provided that $M$ is a bipartite factor of $F\setminus E(T)$ such that for each $v\in V(G)\setminus A$, 
$d_M(v)+d_{F\cap T}(v)\ge m$.
}\end{thm}
\begin{proof}
{Let $B=\{v\in V(G):d_F(v)\ge m+1\}=V(G)\setminus A$. Define $G_0=T\cup F$ and $F'=F\setminus E(M)$.
Let $T_0$ be an $m$-tree-connected factor of $G_0\setminus E(M)$ such that 
(i) for all $v\in A$, $d_{T_0}(v)\le d_T(v)$, and 
(ii) for all $v\in B$, $d_{T_0\cup F'}(v)\le d_T(v)+d_F(v)-m$,
and (iii) for all $v\in V(G)$, $E_{F\cap T}(v)\subseteq E_{T_0}(v)$.
Note that $T$ is a natural candidate for $T_0$,
since for each $v\in B$, by the assumption, we have
$d_{T\cup F'}(v)= d_{T}(v)+d_{F'}(v)-d_{F'\cap T}(v)=d_{T}(v)+(d_{F}(v)-d_M(v))-d_{F\cap T}(v)\le d_T(v)+d_F(v)-m$.
Consider $T_0$ with the maximum $|E(T_0)\cap E(F)|$. 
Now, are ready to prove the following claim.
%===================================	Claim
\vspace{2mm}\\
{\bf Claim.} 
If $vx\in E(F)\setminus E(M)$ and $v\in A$, then $vx\in E(T_0)$.
\vspace{2mm}\\
%===================================
{\bf Proof of Claim.} 
Suppose, to the contrary, that there is an edge $vx\in E(F)\setminus (E(M)\cup E(T_0))$ with $v\in A$.
By item (iii), we must have $x\in B$.
 Otherwise, $vx$ is an edge of $E(F)\setminus E(T)$ with both ends in $A$ which is impossible. 
Let $Q$ be a minimal $m$-tree-connected subgraph of $T_0$ including $v$ and $x$. 
Since $Q$ is $m$-edge-connected, 
$d_Q (v) \ge m$.
On the other hand, $d_{F}(v)\le m$.
Thus there exists an edge $vy \in E(Q)\setminus E(F)$ so that $vy\in E(T_0)\setminus E(F)$
(we might have $x = y$).
Define $T'_0=T_0-vy+vx$.
By Lemma~\ref{lem:xGy-exchange}, the graph $T_0$
is still $m$-tree-connected.
Note that $T'_0$ is a factor of $G_0\setminus E(M)$, because of $vx\in E(F)\setminus E(M)$.
In addition, $d_{T'_0}(u)\le d_{T_0}(u)$ when $u\in \{v,y\}$,
and $d_{T'_0\cup F'}(u)\le d_{T_0\cup F'}(u)$ when $u\in \{x,y\}$.
It is not hard to check that $T'_0$ satisfies items (i), (ii), and (iii).
Since $|E(T_0')\cap E(F)| > |E(T_0)\cap E(F)|$, we derive a contradiction to the maximality of $T_0$.
Hence the claim holds.

Let $H$ be an $m$-tree-connected factor of $G_0$ containing $E(F')$ satisfying 
 $d_{H}(v)\le d_T(v)+\max\{0,d_F(v)-m\}$ for all $v\in V(G)$.
Note that $T_0\cup F'$ is a natural candidate for $H$.
In fact, if $v\in A$, then by the above-mentioned claim, we must have $d_{T_0\cup F'}(v)= d_{T_0}(v)\le d_T(v) = d_T(v)+\max\{0,d_F(v)-m\}$.
Also, if $v\in B$, then by the assumption on $T_0$, we must have $d_{T_0\cup F'}(v)\le d_T(v)+d_F(v)-m$. 
Consider $H$ with the maximum $|E(H)\cap E(M)|$.
We are going to show that $H$ is the desired factor we are looking for.
To prove this, it remains to show that 
 $d_{H}(v)\ge d_F(v)$ for all $v\in V(G)$.
Suppose, to the contrary, that $d_{H}(x)< d_F(x)$ for a vertex $x\in V(G)$.
Thus there exists an edge $vx\in E(F)\setminus E(H)$. 
We may assume that $x\in B$, because $d_{H}(v) \ge m \ge d_F(v)$ for all $v\in A$.
Hence $vx\in E(M)\setminus E(H)$. 
Let $Q$ be the minimal $m$-tree-connected subgraph of $H$ including $vx$.
Since $d_F(v)\le m\le d_{Q}(v)$ and $vx\in E(F)\setminus E(Q)$, there must be an edge $vy\in E(Q)\setminus E(F)$ (we might have $x = y$).
Define $H'=H-vy+vx$.
By Lemma~\ref{lem:xGy-exchange}, the graph $H'$ is still $m$-tree-connected.
Note that $d_{H'}(u)\le d_{H}(u)\le d_T(u)+\max\{0,d_F(u)-m\}$ for all $u\in A$, and 
 $d_{H'}(u)\le d_{H}(u)\le d_T(u)+d_F(u)-m$ for all $u\in B\setminus \{x\}$.
In addition, $d_{H'}(x)\le d_F(x)\le d_T(x)+d_F(x)-m$, because of $d_T(v)\ge m$.
Since $|E(H')\cap E(M)| > |E(H)\cap E(M)|$, we derive a contradiction to the maximality of $H$.
Hence the proof is completed.
}\end{proof}
\begin{remark}
{Note that the results of this section can be developed to an $m$-rigid version instead of the current $m$-tree-connected version using the same process in the proofs. Moreover, these results can be developed to a supermodular version, see \cite{R}.
}\end{remark}
Now, we are ready to prove the main result of this section.
\begin{thm}\label{thm:tree-connected:contains}
{Let $G$ be a graph with an $m$-tree-connected factor $T$. If $G$ has a factor $F$, then it has an $m$-tree-connected factor $H$ such that for each vertex $v$, 
$$d_F(v)\le d_H(v)\le d_T(v)+\max\{0,d_F(v)-m\}.$$
}\end{thm}
\begin{proof}
{By induction on $|E(F)|$. For $|E(F)|=0$, the proof is clear.
We may assume that $|E(F)|>0$.
Put $A=\{v\in V(G):d_F(v)\le m\}$ and $B=\{v\in V(G):d_F(v)\ge m+1\}=V(G)\setminus A$. 
If there exists an edge $xy\in E(F)\setminus E(T)$ with both ends in $A$, 
then by applying the induction hypothesis to the graphs $G'=G$ and $F'=F-xy$, 
there is an $m$-tree-connected factor $H$ such that for each vertex $v$, 
$d_{F'}(v)\le d_{H}(v)\le d_T(v)+\max\{0,d_{F'}(v)-m\}$.
Note that for each $v\in \{x,y\}$, $d_{F}(v)\le m\le d_{H}(v)\le d_T(v)=d_T(v)+\max\{0,d_{F}(v)-m\}$.
It is easy to see that $H$ is the desired factor.
If there exists an edge $xy\in E(F)\setminus E(T)$ with both ends in $B$, 
then by applying the induction hypothesis to the graphs $G'=G-xy$ and $F'=F-xy$, 
there is an $m$-tree-connected factor $H'$ of $G'$ such that for each vertex $v$, 
$d_{F'}(v)\le d_{H'}(v)\le d_T(v)+\max\{0,d_{F'}(v)-m\}$.
Let $H=H'+xy$. If $v\in \{x,y\}$, then we still have $d_{F'}(v)\ge m$ and so
$d_{F}(v)\le d_{H}(v)\le d_T(v)+d_{F}(v)-m$. 
Hence it is easy to see that $H$ is the desired factor.
We may therefore assume that $F\setminus E(T)$ is a bipartite graph with the bipartition $(A,B)$. 
Take $M$ to be the graph $F\setminus E(T)$. For each $v\in B$, 
$d_M(v)+d_{F\cap T}(v)= d_F(v) \ge m$. 
Thus by Theorem~\ref{thm:4.13}, there is a factor $H$ with the desired properties.
}\end{proof}
The following corollary provides an interesting and equivalent version for Theorem~\ref{thm:tree-connected:contains}.
\begin{cor}
{Let $G$ be a graph. If $G$ contains a $(g,f)$-factor and an $m$-tree-connected $(m,f')$-factor, then it has an $m$-tree-connected $(g,f+f'-m)$-factor, where $g$, $f$, $f'$ are nonnegative integer-valued functions on $V(G)$ with $ g\le f$, $m\le f$, and $m\le f'$.
}\end{cor}
\begin{proof}
{Let $F$ be a $(g,f)$-factor and let $T$ be an $m$-tree-connected $(m,f')$-factor of $G$.
By Theorem~\ref{thm:tree-connected:contains}, the graph $G$ has an $m$-tree-connected factor $H$ such that for each vertex $v$, $g(v)\le d_F(v)\le d_H(v)\le d_T(v)+\max\{0,d_F(v)-m\}\le f'(v)+\max\{0,f(v)-m\}= f'(v)+f(v)-m$.
Hence the assertion holds.
}\end{proof}
%
%
%
%
%
%=============================================================================
%
%
%
%
%
%==============================================================================
%

%\bibliographystyle{siam}
%\bibliography{ref}
\end{document}